\documentclass[12pt,a4paper]{article}
\usepackage{amsmath,amssymb,amsthm,fullpage}
\newtheorem{theorem}{Theorem}
\newtheorem{lemma}{Lemma}
\newtheorem{remark}{Remark}
\usepackage{mathptmx}
\DeclareMathOperator{\C}{\textit{C}\,}
\DeclareMathOperator{\K}{\textit{K}\,}
\DeclareMathOperator{\CT}{\textit{CT}\,}
\begin{document}
\title{Game arguments in computability theory and algorithmic information theory}
\author{Andrej Muchnik\thanks{1958--2007} \and Alexander Shen\thanks{LIRMM, CNRS \& Universit\'e Montpellier 2, on leave from IITP RAS, Moscow.\hfil\break \texttt{alexander.shen@lirmm.fr, sasha.shen@gmail.com}. Supported in part by NAFIT EMER-008-01 grant and RFBR 09-01-00709-a grant} 
\and Mikhail Vyugin}

\maketitle

\begin{abstract}
\sloppy
We provide some examples showing how game-theoretic arguments (the approach that goes back to Lachlan and was developed by An.~Muchnik) can be used in computability theory and algorithmic information theory. To illustrate this technique, we start with a proof of a classical result, the unique numbering theorem of Friedberg, translated to the game language. Then we provide game-theoretic proofs for three other results: (1)~the gap between conditional complexity and total conditional complexity; (2)~Epstein--Levin theorem relating a priori and prefix complexity for a stochastic set (for which we provide a new game-theoretic proof) and (3)~some result about information distances in algorithmic information theory (obtained by two of the authors [A.M. and M.V.] several years ago but not yet published). An extended abstract of this paper appeared in~\cite{cie}.

\end{abstract}

\bigskip
It often happens that some result in computability theory or algorithmic information theory is essentially about the existence of a winning strategy in some game. This approach was considered by A.~Lachlan for enumerable sets\footnote{As Lachlan writes in~\cite{lachlan}, ``our reason for studying basic games [the kind of games he defined] is that every theorem of $T(\mathcal{R})$ [elementary theory of enumerable sets] known at the present time can be proved by constructing an effective winning strategy for a suitable basic game.''}; later it was (in different forms) used by An.A.~Muchnik~\cite{muchnik-general,muchnik-game,muchnik-ver}. In Section~\ref{friedberg} we illustrate this approach by showing how a classical result of recursion theory (Friedberg's theorem on unique numberings) can be translated into this language. In Section~\ref{total-conditional} we use game approach to relate total conditional complexity $\CT(x|y)$ (the minimal complexity of a \emph{total} program that maps a condition $y$ to some object $x$) and standard conditional complexity (where the program is not necessarily total). Then in Section~\ref{epstein-levin} we provide a new game-theoretic proof of a recent result of Epstein and Levin~\cite{epstein-levin}. Finally, in Section~\ref{sec:mv} we generalize the result of~\cite{vyugin} and show that for every natural numbers $m,n$ and for every string $x_0$ of sufficiently high complexity one can find strings $x_1,\ldots,x_m$ such that all the conditional complexities $\C(x_i|x_j)$ (for all $i,j$ in $\{0,1,2,\ldots,m\}$ such that $i\ne j$; note that $0$ is allowed) are equal to $n+O(1)$ where the constant in $O(1)$ depends only on $m$ (but not on $n$).

\section{Friedberg's unique numbering}
\label{friedberg}

Our first example is a classical result of R.~Friedberg~\cite{friedberg}: the existence of unique numberings.

\begin{theorem}[Friedberg]\label{th:friedberg}
There exists a partial computable function $F(\cdot,\cdot)$ of two natural variables such that:

\textup{(1)}~$F$ is universal, i.e., every computable function $f(\cdot)$ of one variable appears among the functions $F_n:x\mapsto F(n,x)$;

\textup{(2)}~all the functions $F_n$ are different.
\end{theorem}

\begin{proof}
   The proof can be decomposed in two parts. First, we describe some game and explain why the existence of a (computable) winning strategy for one of the players makes the statement of Friedberg's theorem true. In the second part we construct a winning strategy and therefore finish the proof.

\subsection{Game}

The game is infinite and is played on two boards. Each board is a table with an infinite number of columns (numbered $0,1,2\ldots$ from left to right) and rows (numbered $0,1,2,\ldots$ starting from the top). Each player (we call them \textbf{A}lice and \textbf{B}ob, as usual) plays on its own board. The players alternate. At each move player can fill finitely many cells at her/his choice with any natural numbers (s)he wishes. Once a cell is filled, it keeps this number forever (it cannot be erased).

The game is infinite, so in the limit we have two tables $A$ (filled by Alice) and $B$ (filled by Bob). Some cells in the limit tables may remain empty; other contain natural numbers (one in each cell). The winner is determined by the following rule: Bob wins if
\begin{itemize}
\item for each row in $A$-table there exists an identical row in $B$-table;
\item all the rows in $B$-table are different.
\end{itemize}
\begin{lemma}\label{winning-enough}
    Assume that Bob has a computable winning strategy in this game. Then the statement of Theorem~\ref{th:friedberg} is true.
\end{lemma}

\begin{proof}
    A table represents a partial function of two arguments in a natural way: the number in $i$th row and $j$th column is the value of the function on $(i,j)$; if the cell is not filled, the value is undefined.

    Let Alice fill $A$-table with the values of some universal function (so the $j$th cell in the $i$th row is the output of $i$th program on input $j$). Alice does this at her own pace simulating in parallel all the programs (and ignoring Bob's moves). Let Bob apply his computable winning strategy against the described strategy of Alice. Then his table also corresponds to some computable function $B$ (since the entire process is algorithmic). This function satisfies both requirements of Theorem~\ref{th:friedberg}: since $A$-function is universal, every computable function appears in some row of $A$-table and therefore (due to the winning condition) also in some row of $B$-table. So $B$ is universal. On the other hand, all $B_n$ are different since the rows of $B$-table (containing $B_n$) are different.
\end{proof}

\begin{remark}
     If Alice had a computable winning strategy in our game, the statement of Theorem~\ref{th:friedberg} would be false. Indeed, let Bob fill his table with the values of a universal function that satisfies the requirements of the theorem \textup(ignoring Alice's moves\textup). Then Alice fills her table in a computable way and wins. This means that some row of Alice's table does not appear in Bob's table \textup(so his function is not universal\textup) or two rows in Bob's table coincide \textup(so his function does not satisfy the uniqueness requirement\textup).

So we can try the game approach even not knowing for sure who wins in the game; finding out who wins in the game would tell us whether the statement of the theorem is true or false \textup(assuming that the winning strategy is computable\textup).
\end{remark}

\subsection{Winning strategy}

\begin{lemma}\label{winning}
Bob has a computable winning strategy in the game described.
\end{lemma}

Proving this lemma we may completely forget about computability and just describe the winning strategy explicitly  (this is the main advantage of the game approach). We do this in two steps: first we consider a simplified version of the game and explain how Bob can win in this simplified version. Then we explain what he should do in the full version of the game.

In the simplified version of the game Bob, except for filling $B$-table, may \emph{kill} some rows in it. The rows that were killed are not taken into account when the winner is determined. So Bob wins if the final (limit) contents of the tables satisfies two requirements: (1)~for each row in $A$-table there exists an identical valid (non-killed) row in $B$-table, and (2)~all the valid rows in $B$-table are different. (According to this definition, after the row is killed its content does not matter.)

To win the game, Bob hires a countable number of assistants and makes $i$th assistant responsible for $i$th row in $A$-table. The assistants start their work one by one; let us agree that $i$th assistant starts working at move $i$, so at every moment only finitely many assistants are active.  Assistant starts her work by \emph{reserving} some row in $B$-table not reserved by other assistants, and then continues by copying the current contents of $i$th row of $A$-table (for which she is responsible) into this reserved row. Also at some point the assistant may decide to kill the current row reserved by her, reserve a new row,  and start copying the current content  of $i$th row into the new reserved row. Later in the game she may kill the reserved row again, etc.

The instructions for the assistant determine when to kill the reserved row. They should guarantee that
\begin{itemize}
\item if $i$th row in the final (limit) state of $A$-table coincides with some previous row, then $i$th assistant kills her reserved row infinitely many times (so none of her reserved rows remain active);
\item if it is not the case, i.e., if $i$th row is different from all previous rows in the final $A$-table, then $i$th assistant kills her row only finitely many times (and after that faithfully copies $i$th row of $A$-table into that row).
\end{itemize}
If this is arranged, the valid rows of $B$-table correspond to the first occurences of rows with given contents in $A$-table, so they are all different, and contain all the rows of $A$-table.

The instruction for $i$th assistant: \emph{keep track of the number of rows that you have already killed in some counter $k$; if in the current state of $A$-table the first $k$ positions in $i$-th row are identical to the first $k$ positions of some previous row, kill the current reserved row in $B$-table \textup(and increase the counter\textup); if not, continue copying $i$-th row into the current row.}

Let us see why these instructions indeed have the required properties. Imagine that in the limit state of $A$-table the row $i$ is the first row with given content, i.e., is different from all the previous rows. For each of the previous rows let us select and fix some position (column) where the rows differ, and consider the moment $T$ when these positions reach their final states. Let $N$ be the maximum of the selected columns (in all previous rows). After step $T$ the $i$th row in $A$-table differs from all previous rows in one of the first $N$ positions, so if the counter of killed rows exceeds $N$, no more killings are possible (for this assistant).

On the other hand, assume that $i$th assistant kills her row finitely many times and $N$ is the maximal value of her counter. After $N$ is reached, the contents of $i$th row in $A$-table is always different from the previous rows in one of the first $N$ positions, and the same is true in the limit (since this rectangle reaches its limit state at some moment).

\smallskip
So Bob can win in the simplified game, and to finish the proof of Lemma~\ref{winning} we need to explain how Bob can refrain from killing and still win the game.

Let us say that a row is \emph{odd} if it contains a finite odd number of non-empty cells. Bob will now ignore odd rows of $A$-table and at the same time guarantee that all possible odd rows (there are countably many possibilities) appear in $B$-table exactly once. We may assume now without loss of generality that odd rows never appear in $A$-table: if Alice adds some element in a row making this row odd, this element is ignored by Bob until Alice wants to add another element in this row, and then the pair is added. This makes the $A$-table that Bob sees slightly different from what Alice actually does, but all the rows in the limit $A$-table that are not odd (i.e., are infinite or have even number of filled cells) will get through~--- and Bob separately takes care of odd rows.

Now the instructions for assistants change: instead of killing some row, she should fill some cells in this row making it odd, and ensure that this odd row is new (different from all other odd rows of the current $B$-table). After that, this row is considered like if it were killed (no more changes). This guarantees that all non-odd rows of $A$-table appear in $B$-table exactly once.

Also Bob hires an additional assistant who ensures that all possible odd rows appear in $B$-table: she looks at all the possibilities one by one; if some odd row has not appeared yet, she reserves some row and puts the desired content there. (Unlike other assistants, she reserves more and more rows.) This behavior guarantees that all possible odd rows appear in $B$-table exactly once. (Recall that other assistants also avoid repetitions among odd rows.) Lemma~\ref{winning} and Theorem~\ref{th:friedberg} are proven.
\end{proof}

\begin{remark}
Martin Kummer in his note~\textup{\cite{kummer-friedberg}} observes that the property ``$i$-th enumerable set is different from all preceding ones'' is $\mathbf{0}'$-enumerable and therefore the set of minimal indices can be represented as the range of a limit-computable function. This remark can be used instead of explicit construction, though it is less adapted to the game version.
\end{remark}

\section{Total conditional complexity}
\label{total-conditional}

In this section we switch from the general computability theory to the algorithmic information theory and compare the conditional complexity $\C(x|y)$ and the minimal length of the program of a \emph{total} function that maps $y$ to $x$. The latter quantity may be called ``total conditional complexity'' (see, e.g.,~\cite{bauwens}.) It turns out that total conditional complexity $\CT(x|y)$ can be much bigger than $\C(x|y)$.  But let us recall first the definitions.

The \emph{conditional complexity} of a binary string $x$ relative to a binary string $y$ (a \emph{condition}) is defined as the length of the shortest program that maps $y$ to $x$. The definition depends on the choice of the programming language, and one should select an optimal one that makes the complexity minimal (up to $O(1)$ additive term). When the condition $y$ is empty, we get (unconditional plain) complexity of $x$. See, e.g., \cite{shen-uppsala} for more details. The conditional complexity of $x$ relative to $y$ is denoted by $\C(x|y)$; the unconditional complexity of $x$ is denoted by $\C(x)$.

It is easy to see that $\C(x|y)$ can also be defined as the minimal \emph{complexity} of a program that maps $y$ to $x$. (This definition coincides with the previous one up to $O(1)$ additive term; any programming language that allows effective translations from other programming languages can be used.) But in some applications (e.g., in algorithmic statistics, see~\cite{stat}) we are interested in \emph{total} programs, i.e. programs that terminate on every input. Let us define $\CT(x|y)$ as the minimal complexity of a \emph{total} program that maps $y$ to $x$. In general, this restriction could increase complexity, but how significant could be this increase? It turns out that these two quantities may differ drastically, as the following simple theorem shows (this observation was made by several people independently; the first publication is probably~\cite[Section 6.1]{bauwens}).

\begin{theorem}
     \label{th:total}
For every $n$ there exist two strings $x_n$ and $y_n$ of length $n$ such that $\C(x_n|y_n)=O(1)$ but $\CT(x_n|y_n)\ge n$.
\end{theorem}

\begin{proof}
    To prove this theorem, consider a game $G_n$ (for each $n$). In this game Alice constructs a partial function $A$ from $\mathbb{B}^n$ to $\mathbb{B}^n$, i.e., a function defined on (some) $n$-bit strings, whose values are also $n$-bit strings. Bob constructs a list $B_1,\ldots, B_k$ of total functions of type $\mathbb{B}^n\to\mathbb{B}^n$. (Here $\mathbb{B}=\{0,1\}$.)

The players alternate; at each move Alice can add several strings to the domain of $A$ and choose some values for $A$ on these strings;  the existing values cannot be changed. Bob can add some total functions to the list, but the total length of the list should remain less than $2^n$. The players can also leave their data unchanged; the game, though infinite by definition, is essentially finite since only finite number of nontrivial moves is possible. The winner is determined as follows: Alice wins if in the limit state there exists a $n$-bit string $y$ such that $A(y)$ is defined and is different from all $B_1(y),\ldots,B_k(y)$.

\begin{lemma}
   \label{lemma:CT-C}
   Alice has a computable \textup(uniformly in $n$\textup) winning strategy in this game.\footnote{Since the game is effectively finite, in fact the existence of a winning strategy implies the existence of a computable one. But it is easy to describe the computable strategy explicitly.}
\end{lemma}

Before proving this lemma, let us explain why it proves Theorem~\ref{th:total}. Let (for every~$n$) Alice play against the following strategy of Bob: he just enumerates all the total functions of type $\mathbb{B}^n\to\mathbb{B}^n$ that have complexity less than $n$, and adds them to the list when they appear. (As in the previous section, Bob does not really care about Alice's moves.) The behavior of Alice is then also computable since she plays a computable strategy againt a computable opponent. Let $y_n$ be the string where Alice wins, and let $x_n$ be equal to $A(y_n)$ where $A$ is the function constructed by Alice.

It is easy to see that $\C(x_n|y_n)=O(1)$; indeed, knowing $y_n$, we know $n$, can simulate the game, and find $x_n$ during this simulation. On the other hand, if there were a total function of complexity less than $n$ that maps $y_n$ to $x_n$, then this function would be in the list and Bob would win.

So it remains to prove the lemma by showing the strategy for Alice. This strategy is straightforward: first Alice selects some $y$ and says that $A(y)$ is equal to some $x$. (This choice can be done in arbitrary way, if Bob has not selected any functions yet; we may always assume it is the case by postponing the first move of Bob; the timing is not important in this game.) Then Alice waits until one of Bob's functions maps $y$ to $x$. This may never happen; in this case Alice does nothing else and wins with $x$ and $y$. But if this happens, Alice selects another $y$ and chooses $x$ that is different from $B_1(y),\ldots,B_k(y)$ for all total functions $B_1,\ldots,B_k$ that are currently in Bob's list. Since there are less than $2^n$ total functions in the list, it is always possible. Also, since Bob can make at most $2^{n}-1$ nontrivial moves, Alice will not run out of strings $y$. Lemma~\ref{lemma:CT-C} and theorem~\ref{th:total} are proven.
\end{proof}

A well-known result of Bennett, G\'acs, Li, Vit\'anyi and Zurek~\cite{bglvz} says that if $C(x|y)$ and $C(y|x)$ are small (do not exceed some $k$), there exists a program of complexity at most $k+O(\log k)$ that maps $x$ to $y$ and at the same time maps $y$ to $x$ (given an additional advice bit that says which of these two tasks it should perform). The natural question arises: is a similar statement true for total conditional complexities and computable bijections? The (partly negative) answer is provided by the following theorem (a sketch of its proof is given in~\cite{muchnik-game}, but some important details are missing there):

\begin{theorem}
Let $x$ and $y$ be two binary strings of length at most $n$. Then there exists a program $t$ that computes a permutation of the set of all binary strings,  maps $x$ to $y$ and  
    $$
\C(t)\le \CT(x|y)+\CT(y|x)+O(\log n).    
    $$
This bound cannot be improved significantly: for every $k$ and $n$ such that $n>2k$ there exist two strings $x$ and $y$ of length $n$ such that $\C(x),\C(y)\le k+O(\log n)$ but any program for a bijection that maps $x$ to $y$ has complexity at least~$2k-O(1)$.
\end{theorem}

Note the difference with non-total result mentioned earlier: now instead of \emph{maximum} of $\C(x|y)$ and $\C(y|x)$ we need their \emph{sum}.

\begin{proof}
The first part is simple. Having two total programs $p$ (mapping $x$ to $y$) and $q$ (mapping $y$ to $x$) and knowing $n$, we compute a one-to-one correspondence between two sets of strings of length at most $n$: string $u$ corresponds to $v$ if $p(u)=v$ and $q(v)=u$ at the same time. (This correspondence can be effectively computed as a finite object, since both $p$ and $q$ are total according to our assumption.) Then we extend this correspondence to a permutation of the set of all strings of length at most $n$; one more extension gives a computable permutation of the set of all binary strings (we may assume, for example, that all longer strings are mapped to itself). The progam $t$ obtained in this way can be effectively constructed given $p$, $q$ and $n$, so we get the required bound. (Note that both $\CT(x|y)$ and $\CT(y|x)$ do not exceed $n$, therefore forming a pair from $p$ and $q$ can be done with $O(\log n)$-overhead.)

For the second part, we again consider a game. Let $X$ and $Y$ be sets that contain $2^n$ elements (recall that $n>2k$). Alice can \emph{mark} some elements in $X$ or $Y$, not more than $2^k$ elements in each set. Bob can list (sequentially) some bijections between $X$ and $Y$, at most $2^{2k-2}$ bijections. Winning condition: Bob wins if for every marked element $x\in X$ and for every marked element $y\in Y$ there exists a bijection in the list that maps $x$ to $y$. 

It is easy to see that Bob can win if $2^{2k-2}$ is replaced by $2^{2k}$: when Alice marks new elements, he forms a bijection for every new pair of marked elements, and adds all these bijections to the list; in total there are at most $2^k\cdot 2^k$ such pairs. But $2^{2k-2}$ bijections are not enough:

\begin{lemma}
Alice has a winning strategy in this game.
\end{lemma}

Let us explain why this is enough to prove the theorem. Let $X=Y=\mathbb{B}^n$ (the set of $n$-bit strings). Let Alice play against Bob who generates all programs of complexity less than $2k-2$ and runs them (in parallel) on all elements of $X$; when he finds that some program computes a bijection between $X$ and $Y$, this bijection is added to the list. Since Alice wins, there are some marked elements $x$ and $y$ that are not connected by any bijection in the list. These elements are determined by $n$, $k$, and their ordinal number in the enumeration; the latter can be encoded by $k$ bits since there is at most $2^k$ marked elements in each set (so we get $O(\log n)+k$ bits in total). 

This argument assumes that Alice's strategy is computable given $n$ and $k$; as before, we may note that existence of some strategy implies the existence of a computable one, or look at the actual strategy below.

It remains to show a (computable) winning strategy for Alice. She starts by marking arbitrary elements $x_1\in X$ and $y_1\in Y$ and then waits until Bob provides a bijection that connects them.  After that, Alice chooses (again arbitrarily) some element $x_2\ne x_1$ and waits until $x_2$ is connected with $y_1$ (Bob needs a new bijection for that, since the old one connects $x_1$ and $y_1$). Then Alice switches to $Y$ and chooses a new element $y_2$ not connected to $x_1$, $x_2$ by existing bijections, and waits until Bob adds two new bijections connecting $y_2$ to $x_1$ and $x_2$. Then she continues in the same way, alternating between $X$ and~$Y$. At each step she takes an element not connected by existing bijections to existing elements on the other side. If Alice is able to continue this process, then for each new pair of marked elements a new bijection is needed, so the total number of bijections should be at least $2^{2k}$. 

Things are not so simple, however: it may happen that all elements of $X$ (or $Y$) are already connected to some marked elements\footnote{There are $2^{2k}$ bijections and $2^k$ marked elements, so at most $2^{3k}$ elements can be connected; we know only that $n$ is greater than $2k$, not $3k$.}, so Alice cannot choose $x\in X$ that is not connected to any marked element of $Y$ by any listed bijection. However, Alice can get at least half of new pairs each time.  Indeed, assume that she selects an element in $X$; let us show that she can select an element that is connected to less than half of marked elements in $Y$. Each marked element in $Y$ is connected to at most $2^{2k-2}$ elements in $X$, so the probability that a (uniformly) random element in $X$ is connected to random marked element in $Y$ is at most $1/4$. Therefore, for some element in $X$ only $1/4$ (or less) marked elements in $Y$ are connected to it, and Alice may choose this element. This argument saves at least half of the pairs, so the total number of bijections needed to cover all pairs is at least $2^{2k-1}$, more than Bob has. Lemma is proven.
\end{proof}

\section{Epstein--Levin theorem}
\label{epstein-levin}

In this section we discuss a game-theoretic interpretation of an important recent result of Epstein and Levin~\cite{epstein-levin}. This result can be considered as an extension of some previous observations made by Vereshchagin (see~\cite{stat}). Let us first recall some notions from the algorithmic information theory.

For a finite object $x$ one may consider two quantities. The first one, the \emph{complexity} of $x$, shows how many bits we need to describe $x$ (using an optimal description method). The second one, \emph{a priori probability} of $x$, measures how probable is the appearance of $x$ in a (universal) algorithmic random process. The first approach goes back to Kolmogorov while the second one was suggested earlier by Solomonoff.\footnote{Solomonoff also mentioned complexity as a technical tool somewhere in his paper.} The relation between these two notions in a most clean form was established by Levin and later by Chaitin (see~\cite{history} for more details). 

For that purpose Levin modified the notion of complexity and introduced \emph{prefix complexity} $\K(x)$ where programs (descriptions) satisfy an additional property: if $p$ is a program that outputs $x$, then every extension of $p$ (every string having prefix $p$) also outputs $x$. (Chaitin used another restriction: the set of programs should be prefix-free, i.e., none of the programs is a prefix of another one; though it is a significantly different restriction, it leads to the same notion of complexity up to $O(1)$ additive term.) 

The notion of a priori probability can be formally defined in the following way. Consider a randomized algorithm $M$ without input that outputs some natural number and stops. The output number depends on the internal random bits (fair coin tosses) by~$M$. For every $x$ there is some probability $m_x$ to get $x$ as output. The sum $\sum m_x$ does not exceed $1$; it can be less if the machine $M$ performs a non-terminating computation with positive probability. In this way every machine $M$ corresponds to some function $x\mapsto m_x$. There exists a \emph{universal} machine $M$ of this type, i.e., the machine for which function $x\mapsto m_x$ is maximal up to a constant factor. For example, $M$ can start by choosing a random machine in such a way that every choice has positive probability, and then simulate the chosen machine. We now fix some universal machine $M$ and call the probability $m_x$ to get $x$ on its output \emph{a priori probability} of $x$.

The relation between prefix complexity and a priory probability is quite close: Levin and Chaitin have shown that $\K(x)=-\log_2 m_x +O(1)$. However, the situation changes if we extend prefix complexity and a priori probability to sets. Let $X$ be a set of natural numbers. Then we can consider two quantities that measure the difficulty of a task ``produce some element of $X$'': 
\begin{itemize}
\item \emph{complexity} of $X$, defined as the minimal length of a program that produces some element in $X$;
\item \emph{a priori probability} of $X$, the probability to get some element of $X$ as an output of the universal machine $M$. 
\end{itemize}
As we have mentioned, for singletons the complexity coincides with the minus logarithm of a priori probability up to $O(1)$ additive term. For an arbitrary set of integers this is no more the case: complexity can differ significantly from the minus logarithm of a priori probability. In other words, for an arbitrary set $X$ the quantities
     $$
\max_{x\in X} m_x \qquad \text{and}\qquad \sum_{x\in X} m_x
     $$
(the first one corresponds to the complexity of $X$, the second one is a priori probability of $X$) could be very different. For example, if $X$ is the set of strings of length $n$ that have complexity close to $n$, the first quantity is rather small (since all $m_x$ are close to $2^{-n}$ by construction) while the second one is quite big (a string chosen randomly with respect to the uniform distribution on $n$-bit strings, has complexity close to $n$ with high probability).

Epstein--Levin theorem says that such a big difference is \emph{not} possible if the set $X$ is \emph{stochastic}. The notion of a stochastic object was introduced in the algorithmic statistics.  A finite object $X$ (in our case, a finite set of strings) is called \emph{stochastic} if, informally speaking, $X$ is a ``typical'' representative of some ``simple'' probability distribution. This means that there exist a probability distribution $P$ with finite domain (containing $X$) and rational probabilities such that (1)~$P$~has small complexity, and (2)~the randomness deficiency of $X$ with respect to $P$, defined as $-\log P(X)-\K(X|P)$, is small. (Note that here we speak about complexity of $X$ and $P$, where $X$ is a finite set of strings, and $P$ is a distribution on finite sets of strings. These notions are well defined, since the complexity of a finite object does not depend on the choice of its computable encoding, up to $O(1)$ additive term.) Here $\K(X|P)$ stands for \emph{conditional} prefix complexity of $X$ given $P$, see~\cite{shen-uppsala} for details.

\smallskip
Epstein--Levin theorem is essentially a result about some type of games (we call them \emph{Epstein--Levin games}). To define such a game, fix a finite bipartite graph $E\subset L\times R$ with left part $L$ and right part $R$. A probability distribution $P$ on $R$ with rational values is also fixed, as well as three parameters: some natural number $k$, some natural number $l$ and some positive rational number $\delta$.  After all these objects are fixed, we consider the following game.

Alice assigns some rational \emph{weights} to vertices in $L$. Initially all the weights are zeros, but Alice can increase them during the game. The total weight of $L$ (the sum of weights) should never exceed~$1$. Bob can \emph{mark} some vertices on the left and some vertices on the right. After a vertex is marked, it remains marked forever. The restrictions for Bob: he can mark at most $l$ vertices on the left, and the total $P$-probability of marked vertices on the right should be at most~$\delta$. The winner is determined as follows: Bob wins if every vertex $y$ on the right for which the (limit) total weight of all its $L$-neighbors exceeds $2^{-k}$, either is marked itself (at some point), or has a marked (at some point) neighbor.

Evidently, the task of Bob becomes harder if $l$ or $\delta$ decrease (he has less freedom in marking vertices), and becomes easier if $k$ decreases (he cares about less vertices). So the greater $k$ and the smaller $\delta$ is, the bigger $l$ is needed by Bob to win. The following lemma gives a bound (with some absolute constant in $O$-notation):

\begin{lemma}\label{lem:el}
For $l=O(2^{k}\log (1/\delta)$ Bob has a computable winning strategy in the described game.
\end{lemma}

Before proving this lemma, let us explain the connection between this game and the statement of Epstein--Levin theorem. Vertices in $R$ are finite sets of integers; vertices in $L$ are integers, and the edges correspond to $\in$-relation. Alice's weights are a priori probabilities of integers (more precisely, increasing approximations to them). The distribution $P$ on $R$ is a simple distribution (on a finite family $R$ of finite sets) that is assumed to make $X$ (from Levin--Epstein theorem) stochastic. Bob may mark $X$, but this would make it non-random with respect to $P$ (marked vertices form a $P$-small subset and therefore all have big randomness deficiency), so Epstein and Levin do not need to care about $X$ any more. If $X$ is not marked and has big total weight (= the total a priori probability), $X$ is guaranteed to have a marked neighbor. This means that some element of $X$ is marked and therefore has small complexity (since there are only few marked elements); this is what Epstein--Levin theorem says. (Of course, one needs to use some specific bounds instead of ``small'' and ``large'' etc., we provide the exact statements after the proof of the lemma.)

\begin{proof}
To prove the existence of a winning strategy for Bob, we use the following (quite unusual) type of argument: we exhibit a simple \emph{probabilistic} strategy for Bob that guarantees some positive probability of winning against any strategy of Alice. Since the game is essentially a finite game with full information (see the comments at the end of the proof about how to make it really finite), either Alice or Bob have a winning strategy. And if Alice had one, no probabilistic strategy for Bob could have a positive probability of winning.

Let us describe this strategy for Bob. It is rather simple: if Alice increases weight of some vertex $x$ in $L$ by an additional $\varepsilon>0$, Bob responds by tossing a coin and marking $x$ with probability $c2^k\varepsilon$, while $c>1$ is some constant to be chosen later.  We need also to specify what Bob does if $c2^k\varepsilon>1$ (this always happens if $\varepsilon$ is $2^{-k}$ or more). In this case Bob marks $x$ for sure. Note also that without loss of generality we may assume that Alice increases weights one at a time, since we can split her move into a sequence of moves.

We have explained how Bob marks $L$-vertices; if at some point this does not help for some $R$-vertex, i.e., this vertex has total weight at least $2^{-k}$ but no marked neighbors, Bob immediately marks this $R$-vertex (as well as all other vertices with this property).

The probabilistic strategy for Bob is described, and we need to consider some (deterministic) strategy~$\alpha$ for Alice and show that the probability of winning the game for Bob (for suitable $c$, see below about the choice of $c$) is positive when playing against~$\alpha$. By construction, there are two reasons why Bob could lose the game:
\begin{itemize}
\item the total measure of marked $R$-vertices exceeds $\delta$; 
\item the number of marked $L$-vertices exceeds $l$.
\end{itemize}
To show that with positive probability none of this events happen, we ensure that probability of each event is less than $1/2$. For that we show that the expected $P$-measure of marked $R$-vertices is less than $\delta/2$ and the expected number of marked $L$-vertices is less than $l/2$.

Let us fix some~$y$ and estimate the probability for $y$ to be marked by Bob (=~to have no marked neighbors when the sum of weights of $y$'s neighbors achieves $2^{-k}$). Assume that the weights of neighbors of $y$ were increased by $\varepsilon_1,\ldots,\varepsilon_u$ during the game, and now $\sum \varepsilon_i\ge2^{-k}$. After each increase the corresponding neighbor of $y$ was marked with probability $c2^k\varepsilon_i$, so the probability that all the neighbors remain not marked, does not exceed
    $$
(1-c2^k\varepsilon_1)\cdot\ldots\cdot (1-c2^k\varepsilon_u)\le e^{-c2^k(\varepsilon_1+\ldots+\varepsilon_u)}\le e^{-c}
    $$
(recall that $(1-t)\le e^{-t}$ and that $\sum\varepsilon_i\ge 2^{-k}$). Therefore for every measure $P$ the expected  $P$-measure of marked vertices on the right (the weighted average of numbers not exceeding $e^{-c}$) does not exceed $e^{-c}$. So it is enough to let $c$ be $\ln (1/\delta)+O(1)$.

In fact, this picture is oversimplified: the estimate for probability should be done more carefully, since the values of $\varepsilon_1,\ldots,\varepsilon_u$ may depend on Bob's moves. The situation can be described as follows: our opponent (following some probabilistic strategy) tells us some numbers in $[0,1]$ (one by one). After the opponent names some $\varepsilon$, we perform random coin tossing with probability of success $\varepsilon$. Then for every $t$ the probability of the event ``at the moment when the sum of numbers exceeds $t$, we still have no successful trials'' does not exceed $e^{-t}$.  (To prove this statement formally, we need a backward induction in the tree of possibilities.)

The expected number of marked $L$-vertices can be estimated in the same way. Here the opponent also gives us some numbers whose sum is guaranteed not to exceed some $t$ ($t=c2^k$ in our case), and we use them as probabilities of success for random coin tosses. Similar argument shows that the expected number of successes does not exceed $t$. We need $t=c2^k <l/2$, so we take $l=c2^{k+2}=2^{k+2}(\ln (1/\delta)+O(1)) = O(2^k\log(1/\delta))$.

To finish the proof of the lemma, one last remark is needed. To make our arguments (a transition from a probabilistic strategy to a deterministic one) correct, we need to make the game finite. One may assume that current weights of vertices on the left all have the form $2^{-m}$ for some integer $m$ (replacing weights by approximations from below, we can compensate for an additional factor of $2$ by changing $k$ by $1$). Still the game is not finite, since Alice can start with very small weights. However, this is not important: the graph is finite, and all very small weights can be replaced by some $2^{-m}$. If $2^{-m}\cdot\#L<1$, then the sum of weights still does not exceed $2$, and this again is a constant factor.

\end{proof}

Now we can apply this Lemma to prove Epstein--Levin theorem. Let us first give exact definitions. A finite object $X$ is called \emph{$\alpha$-$\beta$-stochastic} if there exists a finite probability distribution $P$ (with finite support and rational values, so it is a finite object) such that
\begin{itemize}
\item $\K(P)$ does not exceed $\alpha$;
\item the deficiency $d(X|P)$, defined as $-\log P(X)-\K(X|P)$, does not exceed $\beta$.
\end{itemize}

\begin{theorem}[Epstein--Levin]
If a finite set $X$ is $\alpha$-$\beta$-stochastic, and its total a priori probability $\sum_{x\in X} m_x$ exceeds $2^{-k}$, then $X$ contains some element $x$ such that
     $$
\K(x)\le k+\K(k)+\log K(k)+\alpha+O(\log\beta)+O(1).
     $$
\end{theorem}

The sum $\sum_{x\in X} m_X$ can be called \emph{a priori} probability of the problem ``produce some element of $X$'', and $\min_{x\in X} \K(x)$ can be called prefix complexity of the same problem. The Epstein--Levin theorem guarantees that for $\alpha$-$\beta$-stochastic sets $X$ with small $\alpha$ and $\beta$ the prefix complexity is logarithmically close to the minus logarithm of a priori probability.

\begin{proof}
We follow the plan outlined above. Let $P$ be the finite probability distribution that makes $X$ stochastic. This means that $\K(P)\le \alpha$ and $d(X|P)=-\log P(X)-\K(X|P)\le\beta$. Consider Epstein--Levin game where $R$ is the support of $P$, the left-hand side  $L$ is the union of all sets in $R$ and edges connect each set $U\in R$ to all its elements. To describe the game completely, we need to specify parameters $k$, $l$, and $\delta$. The parameter $k$ is taken from the statement of our theorem; $\delta=2^{-d}$ where $d$ will be chosen later, and $l=O(2^k\log(1/\delta))=O(d2^k)$ is determined by $k$ and $d$ as described in Lemma~\ref{lem:el}. (This guarantees that Bob has a winning strategy in the game.) Then we let Bob play in this game against Alice who assigns (in the limit) weight $m_x$ to every element $x\in L$.

We will choose $d$ in such a way that all marked elements in $R$ have deficiency greater that $\beta$; our assumptions then guarantee that $X$ is not marked. Lemma~\ref{lem:el} then guarantees that $X$ has a marked neighbor, i.e., that some element of $X$ is marked. It remains to estimate the complexity of marked elements in $L$.

Why marked elements in $R$ have high deficiency? We know that the total measure of marked elements in $R$ does not exceed $2^{-d}$. Consider the semimeasure $P'$ that equals $2^d P$ on marked elements and $0$ otherwise; $P'$ can be enumerated if $P$, $d$, and $k$ are given, so
     $$\K(U|P,d,k)\le -\log P'(U)+ O(1)$$
for every $U$ in $R$. If $U$ is not marked, this is trivial (the right hand side is infinite); for marked $U$ we have
     $$\K(U|P,d,k)\le -\log P(U) - d + O(1)$$ 
and therefore
     $$\K(U|P)\le -\log P(U) - d + \K(d) + \K(k)+O(1),$$
so 
     $$d(U|P)\ge d - \K(d) - \K(k)-O(1)$$             
for all marked $U$ in $R$. So wee need the inequality 
     $$d - \K(d) - \K(k)-O(1)>\beta$$
to ensure that $X$ is not marked.  This is guaranteed for sure if 
     $$d = 2(\beta+\K(k))+O(1)$$
(we do not care about constant factor in $d$ since only $\log d$ will be used in the complexity bound below).   

After $d$ is chosen, we need to estimate the complexity of marked elements in $L$. They can be enumerated given $P$, $k$, $d$ and there is at most $O(2^k d)$ of them, so for every marked $x\in L$ we have
    $$\K(x|P,k,d)\le k+\log d+O(1)$$
and
    $$\K(x)\le \K(P)+\K(k,d)+k+\log d+O(1).$$
Recalling that $\K(P)\le \alpha$ and $d=2(\beta+\K(k))+O(1)$, we get
\begin{multline*}
    $$\K(x)\le \alpha + \K(k,\K(k),\beta)+k+\log \beta + \log \K(k)+O(1)\le\\\le \alpha+\K(k,\K(k))+\K(\beta)+k+\log\beta+\log\K(k)+O(1);
\end{multline*}    
it remains to note that $\K(k,\K(k))=\K(k)$ and that $\K(\beta)=O(\log \beta)$.
\end{proof}

\section{Information distance}\label{sec:mv}

Consider the following problem. Let $m$ be some constant. Given a string $x_0$ and integer $n$, we want to find strings $x_1,\ldots,x_m$ such that $\C(x_i|x_j)=n+O(1)$ for all pairs of different $i,j$ in the range $0,\ldots,m$. (Note that both $i$ and $j$ can be equal to $0$). This is possible only if $x_0$ has high enough complexity, at least~$n$, since $\C(x_0|x_j)$ is bounded by $\C(x_0)$. It turns out that such $x_1,\ldots,x_m$ indeed exist if $C(x_0)$ is high enough (though the required complexity of $x_0$ is greater than $n$), and the constant hidden in $O(1)$-notation does not depend on $n$ (but depends on $m$).

This statement is non-trivial even for $n=1$: it says that for every $n$ and for every string $x$ of high enough complexity there exists a string $y$ such that both $\C(x|y)$ and $\C(y|x)$ are equal to $n+O(1)$. This special case was considered in~\cite{vyugin}, the condition there is $\C(x_0)>2n$ (which is better than provided by our general result). Later~\cite{topology} a different technique (using some topological arguments) was used to improve this result and show that $\C(x_0)>n+O(\log n)$ is enough.

Here is the exact statement that specifies also the dependence of $O(1)$-constant on~$m$:

\begin{theorem}
For every $m$ and $n$ and for every binary string $x_0$ such that
    $$
\C(x_0)> n(m^2+m+1)+O(\log m)
    $$
there exist strings $x_1,\ldots,x_m$ such that
     $$
n\le C(x_i|x_j) \le n+O(\log m)
     $$
for every two different $i,j\in\{0,\ldots,m\}$.
\end{theorem}

Note that the high precision is what makes this theorem non-trivial (if an additional term $O(\log  C(x_0))$ were allowed, one could take the shortest program for $x_0$ and replace first $n$ bits in it by $m$ independent random strings).

\begin{proof}
Let us explain the game that corresponds to this statement. It is played on graph with $(m+1)$ parts $X_0,\ldots,X_m$. There are countably many vertices in each part $X_i$ (representing possible values of $x_i$); we will assume that all $X_i$ are disjoint copies of the set $\mathbb{B}^*$ of all binary strings. As usual, there are two players: Alice and Bob. Alice may connect vertices from different parts by \emph{undirected} edges, while Bob can connect them by \emph{directed} edges. Alice and Bob make alternating moves; at each move they can add any finite set of edges. Alice can also \emph{mark} vertices $x_0$ in $X_0$. The restrictions are:
\begin{itemize}
\item Alice may mark at most $m2^{n+1+nm(m+1)}$ vertices (in $X_0$);
\item for each vertex $x_i\in X_i$ and for each $j\ne i$,  Alice may have at most $m(m+1)2^n$ undirected edges connecting $x_i$ with vertices in $X_j$;
\item for each vertex $x_i\in X_i$ and for each $j\ne i$, Bob should have less than $2^n$ \emph{outgoing} edges from $x_i$ to vertices in $X_j$. (Note that the number of \emph{incoming} edges is not bounded.)
\end{itemize}

The game is infinite. Alice wins if (in the limit) for every non-marked vertex $x_0\in X_0$ there exist vertices $x_1,\ldots,x_m$ from $X_1,\ldots,X_m$ such that every two vertices $x_i,x_j$ (where $i\ne j$) are connected by an undirected (Alice's) edge, but not connected by a directed (Bob's) edge.
\begin{lemma}\label{lem:mv}
   Alice has a computable winning strategy in this game.
\end{lemma}

It is easy to see how this lemma can be used to prove the statement. Imagine that Bob draws an edge $x_i\to x_j$ when he discovers that $\C(x_j|x_i)<n$. Then he never violates the restriction. Alice can computably win against this strategy; every marked vertex then has small complexity, since a marked vertex can be described by its ordinal number in the enumeration order. This ordinal number requires 
     $$\log (m2^{n+1+mn(m+1)}) = \log m + O(1) + n+m^2n +nm$$
bits, and to describe the game we need additional $O(\log n)+O(\log m)$ bits to specify $m$ and $n$, so we get $$C(x_0)\le n(1+m+m^2)+O(\log m)+O(\log n).$$ We want to conclude that $x_0$ is not marked (since it has high complexity), but the bound we have is slightly weaker than needed, it has additional term $O(\log n)$. To get rid of this term, we note that (for given $m$) the bounds for the number of marked vertices grow exponentially with $n$, so we can describe all marked vertices (for given $m$ and for all $n$) simultaneously, and the overhead in the complexity caused by marked vertices for smaller values of $n$ is bounded by~$O(1)$.

For every non-marked vertex $x_0$ there exist $x_1,\ldots,x_m$ that satisfy the winning conditions. For them $\C(x_j|x_i)\ge n$ (otherwise Bob would connect them by a directed edge), and $\C(x_j|x_i)\le n+O(\log m)$, since $x_j$ can be obtained from $x_i$ if we know $i$, $j$, and the ordinal number of undirected edge  $x_i$--$x_j$ among all the edges that connect $x_i$ to $X_j$, in the order of appearance of those edges in the game.

\smallskip
So it remains to prove the lemma. To make clear the idea of the proof, let us first consider the case $m=1$. In this case we deal with two countable sets $X_0$ and $X_1$, Alice's degree is bounded by $2^{n+1}$ and the total number of marked vertices should not exceed $2^{3n+1}$. To explain Alice's strategy, let us tell a story first. 

Imagine a ``marriage agency'' whose business is to form pairs $(x_0, x_1)$ of elements $x_0\in X_0$ and $x_1\in X_1$. After a pair is formed (or at some later moment), each of the ``partners'' (elements of the pair) may ``complain'' about the other one. Then the pair is dissolved and both elements become free. Later agency can try them with new partners.

The mission of the agency is to provide stable pairs for everybody or almost everybody. Of course, this is not always possible: imagine that some element complains about all partners. Moreover, even if additionally require that each element makes less than $M$ complaints, it may happen that for some $x$ all its partners complain about $x$ (still making less than $M$ complaints each), and the agency cannot do much for~$x$.	

\smallskip
However, by clever planning the agency can control the damage and ensure that
\begin{itemize}
\item agency makes at most $2M$ attempts to find a partner for any given element (never trying the same partnership twice);
\item all elements of $X_0$, except for at most $2M^3$ ``hopeless'' ones, ultimately get a stable partnership, and hopeless elements are explicitly marked.
\end{itemize}

(Note that the last requirement treats $X_0$ and $X_1$ in a non-symmetrical way.)

\smallskip
The agency can achieve its goals using the following strategy. First it chooses an arbitrary bijection between $X_0$ and $X_1$ and creates all corresponding pairs. Then it treats complaints one by one: if some $x_0$ complains about its current partner $x_1$ or vice versa, the pair $(x_0,x_1)$ is dissolved. Then agency tries to find a new partner for $x_0$ among elements of $X_1$ \emph{with matching experience}. 

The last requirement is the crucial point of our argument: it means that in the new pair \emph{the number of complaints made by one partner should be equal to the number of complaints received by the other one.} In this way an unlucky element who was rejected $M-1$ times will get a partner who made $M-1$ complaints and therefore is unable to complain again. So nobody will be rejected $M$ or more times. 

The bad news is that sometimes for an element $x_0$ from a dissolved pair there is no partner with matching experience; in this case $x_0$ is declared ``hopeless'' and never considered again. We should estimate the maximal number of hopeless elements. We can encode ``experience'' as a pair of two integers in range $[0,M)$, so there are at most $M^2$ possible values of this parameter, and hopeless elements can be divided into $M^2$ classes. Let us show that in each class there are at most $2M$ elements. Since elements in $X_0$ and $X_1$ change their experience simultaneously (when a complaint is made), and newly formed pairs are made of matching elements, free elements in $X_1$ also form $M^2$ classes of the same cardinalities. If there are already $2M$ hopeless elements in some class, there are also $2M$ matching free elements. New hopeless element in this class cannot appear since one of there matching free elements can be used to form a new pair. (Recall that each element can send less than $M$ complaints and receive less than $M$ complaints, so one of the $2M$ free elements of matching experience was not tried yet.)

One last remark about the agency's strategy: we started with making infinitely many pairs (using some bijection between $X_0$ and $X_1$) at once. It is not important, since actual implementation of this decision can be made gradually (we think about some pairs as existing, but they are not yet informed about that).

Now we explain how this story can be transformed into Alice's strategy in the game described. The parameter $M$ (bound for the number of complaints) is $2^n$; then $2M$ equals $2^{n+1}$ and $2M^3$ equals $2^{3n+1}$, as the lemma requires for $m=1$. When agency makes a pair, Alice draws an (undirected) edge between elements of the pair. When the pair is dissolved, an edge (of course) does not disappear, but Alice does not care about it any more, considering only ``active'' edges (that correspond to currently existing pairs). When Bob draws a (directed) edge  $x\to y$ that is parallel to one of the active edges (the undirected edge $x$--$y$), the agency sees that $x$ complains about $y$ (and, according to this complaint, dissolves the pair $x$--$y$). When Bob draws an edge that is not parallel to an active edge, this edge is ignored until parallel active edge appears (corresponding pair is established); then this old edge becomes a complaint and the newly formed pair is dissolved. (If Bob draws an edge that is parallel to an old inactive edge of Alice, this edge never will change anything.) Finally, agency's declaration that some $x_0\in X_0$ is hopeless means that Alice marks~$x_0$.

It is easy to see that the agency's behavior described above can be transformed into Alice's strategy, so Alice indeed has a (computable) winning strategy for the case $m=1$.

After these preparations let us consider the general case. The idea remains the same, but instead of two sets $X_0$ and $X_1$ we now have $m+1$ components $X_0,X_1,\dots,X_m$. Instead of pairs, we have now cliques made of $m+1$ elements, one per component. A participant of a clique may complain about some other participant, and in this case the clique is dissolved (and an attempt to create a new clique for the $X_0$-element of the dissolved one is performed~--- again $X_0$ gets a preferential treatment).

The clique is represented by Alice's edges between all its elements, $m(m+1)/2$ edges in total. A directed Bob's edge $x_i\to x_j$ that connects two elements $x_i$ and $x_j$ of one of the currently active cliques, is understood as a ``complaint'' of $x_i$ againts $x_j$. (Other edges created by Bob are delayed complaints, as before).

The important change is how the ``experience'' is defined. Each vertex remembers $m(m+1)$ non-negative integers corresponding to ordered pairs $(i,j)$. This tuple $I=\{I_{p,q}\}$  (where $p,q\in \{0,1,\ldots,m\}$ and $p\ne q$) is called an ``index'' of a vertex. When $x_i$ complains about $x_j$ (both are elements of the same clique $(x_0,\ldots,x_m)$), all participants of this clique note this and increase $(i,j)$-component of their index (initially filled with zeros) before the clique is dissolved. Note the difference: now each element $x_i$ knows not only how many complaints it made ($I_{ij}$ is the number of complaints about $X_j$-elements) or received ($I_{ji}$ is the number of complaints received from $X_j$-elements), but also the number of complaints between other components (where $x_i$ is only a witness).

After one elements of a clique complains about another one, all elements of the clique update their indices, and the clique is dissolved. To find the new clique for the element $x_0\in X$ from the dissolved clique, we search for free elements with the same index in all the components. Moreover, it is needed that these elements never have sent complaints about each other (but it is OK if some of them were in the same clique, later dissolved because of some other complaint). If this is possible, a new clique is formed; if not, $x_0$ becomes marked (``hopeless'') and other elements of the dissolved clique remain free (outside the cliques).

Since only elements with the same index are combined into cliques, and the indices are updated synchronously, the number of free elements (that do not belong to active clique) is the same for all components (in general and for each value of the index). Note also that all the numbers in the indices are less than $2^n$ (since each of them is a number of complaints sent by some $x_i$ to some $X_j$). When element changes the clique, its index increases along some coordinate, so the number of changes is at most $m(m+1)2^n$, and each change creates $m$ new edges adjacent to this element (one per component). So for every element $x_i$ and for each $j$ there are at most $m(m+1)2^n$ undirected edges that connect $x_i$ to vertices in $X_j$.

To finish the proof of Lemma~\ref{lem:mv}, it remains to prove the bound for the number of marked vertices (=~hopeless elements in $X_0$). For that we estimate the number of marked vertices of each index (recall that the number of possible indices is bounded by $2^{nm(m+1)}$ since its components are less than $2^n$). The idea here is simple: if we have many (at least $2m2^n$) free vertices of some index, we can always find a clique (made of them) for every vertex $x_0\in X_0$ of that index that lost its old clique. Indeed, we find clique elements sequentially in $X_1,\ldots, X_m$; at every step we can find a vertex that has no complaints about already selected vertices and vice versa, since the number of complaints in both directions is less than $2\cdot 2^n$ for each of the components (less than $2^n$ for each direction), and in total less than $2m2^n$ elements in the next component are unusable due to previous ones.
\end{proof}

\section*{Acknowledgments}

Authors are grateful to Leonid Levin, Peter G\'acs, Bruno Bauwens,  the participants of Kolmogorov seminar (Mos\-cow) and all their colleagues in LIRMM (Montpellier) and LIAFA (Paris); special thanks to Rupert H\"olzl for explaining Friedberg's argument.  Robert Soare informed us (at CiE2012, where the preliminary version of this paper~\cite{cie} was presented) about A.H.~Lachlan's paper~\cite{lachlan} where Lachlan initiated the game approach to computability theory (in slightly different context related to enumerable sets); see also~\cite{kummer-transactions}. Lance Fortnow showed us a proof of Friedberg theorem due to Kummer~\cite{kummer-friedberg}.


\begin{thebibliography}{9}

\bibitem{bauwens}
Bruno Bauwens, \emph{Computability in statistical hypotheses testing, and characterizations of independence and directed influences in time series using Kolmogorov complexity}. PhD thesis, Ugent, May 2010.

\bibitem{bglvz}
Charles H. Bennett, P\'eter G\'acs, Ming Li, Paul~M.B.~Vit\'anyi, Wojciech H. Zurek, Thermodynamics of computation and information distance, \emph{Proceedings of 25th ACM STOC}, p.~21--30 (1993). 

\bibitem{history}
Laurent Bienvenu, Alexander Shen, \emph{Algorithmic information theory and martingales},
preprint, arXiv:0906.2614v1 (2009).

\bibitem{epstein-levin}
Samuel Epstein, Leonid A.~Levin, \emph{Sets Have Simple Members},\\
preprint, arXiv:1107.1458v5 (2011).

\bibitem{friedberg}
R.~Friedberg, Three theorems on recursive numberings,
\emph{J. of Symbolic Logic}, \textbf{23}, 309--316 (1958).

\bibitem{kummer-friedberg}
Martin Kummer, An easy priority-free proof of a theorem of Friedberg, \emph{Theoretical Computer Science}, \textbf{74}, 249--251 (1990).

\bibitem{kummer-transactions}
Martin Kummer, The complexity of recursion theoretic games, \emph{Transactions of the American Mathematical Society}, \textbf{358}(1), 59--86 (2005)

\bibitem{lachlan}
A.H.~Lachlan,  On some games which are relevant to the theory of recursively enumerable sets, \emph{Annals of Mathematics}, \textbf{91}(2), 291--310 (1970)

\bibitem{muchnik-general}
Andrej A. Muchnik, On the basic structures of the descriptive theory of algorithms,
\emph{Soviet Math. Dokl.}, \textbf{32}, 671--674 (1985).

\bibitem{muchnik-game}
Andrej A.~Muchnik, Ilya Mezhirov, Alexander Shen, Nikolay Vereshchagin,
\emph{Game interpretation of Kolmogorov complexity}, preprint arXiv:1003.4712v1 (2010).

\bibitem{muchnik-ver}
Nikolay Vereshchagin, Kolmogorov complexity and Games,
\emph{Bulletin of the European Association for Theoretical Computer Science}, \textbf{94}, Feb.~2008, 51--83.

\bibitem{topology}
Andrei Romashchenko, Alexander Shen, Topological arguments for Kolmogorov complexity, \emph{Proceedings of AUTOMATA and JAC 2012 conference}, EPTCS, v.~90, p.~127--132 (2012). 

\bibitem{shen-uppsala}
Alexander Shen, \emph{Algorithmic Information Theory and Kolmogorov Complexity},
Technical Report, Uppsala University, TR2000-034\\ (www.it.uu.se/research/publications/reports/2000-034/).

\bibitem{cie}
Alexander Shen, Game arguments in computability theory and algorithmic information theory, 
In: Barry Cooper, Anuj Dawar, and Benedikt L\"owe, editors, \emph{Computability in Europe 2012 Proceedings}, volume 7318 of Lecture Notes in Computer Science, pages 655--666. Springer, 2012

\bibitem{stat}
Nikolai K.~Vereshchagin, Paul M.B.~Vitanyi, Rate Distortion and Denoising of Individual Data Using Kolmogorov Complexity,
\emph{IEEE Transactions on Information Theory}, \textbf{56}(7), 3438--3454 (July 2010).

\bibitem{vyugin}
Mikhail Vyugin, Information distance and conditional complexities, \emph{Theoretical Computer Science}, v.~271, no.~1--2, p.~145--150 (2002)


\end{thebibliography}
\end{document}